\documentclass[leqno,12pt]{scrartcl}
\usepackage{amssymb, amsmath, amsthm, txfonts, braket, mathrsfs}

\usepackage[all]{xy}

\usepackage{color}

\makeatletter \newcommand{\dashedrightarrow}[1][2pt]{%
\settowidth{\@tempdima}{$\rightarrow$}\rightarrow
\makebox[-\@tempdima]{\hskip-1.5ex\color{white}\rule[0.5ex]{#1}{1pt}}
\phantom{\rightarrow}
} 
\makeatother

\theoremstyle{plain}
\newtheorem{theorem}{Theorem}[section]
\newtheorem*{theorem*}{Theorem}

\newtheorem{proposition}[theorem]{Proposition}
\newtheorem{conjecture}[theorem]{Conjecture}
\newtheorem{lemma}[theorem]{Lemma}
\newtheorem{corollary}[theorem]{Corollary}

\theoremstyle{remark}

\newtheorem*{remark}{\/Remark}

\theoremstyle{definition}
\newtheorem{definition}[theorem]{Definition}

\newcommand{\alb}{\operatorname{alb}}

\newcommand{\ab}{\mathrm{ab}}

\newcommand{\m}{\mathfrak{m}}

\newcommand{\Gal}{\operatorname{Gal}}

\newcommand{\Z}{\mathbb{Z}}

\newcommand{\Qp}{\mathbb{Q}_{p}}

\newcommand{\kbar}{\overline{k}}

\newcommand{\inj}{\hookrightarrow}

\renewcommand{\Im}{\operatorname{Im}}
\newcommand{\Coker}{\operatorname{Coker}}
\newcommand{\Ker}{\operatorname{Ker}}
\newcommand{\isomto}{\stackrel{\simeq}{\longrightarrow}}

\newcommand{\onto}[1]{\stackrel{#1}{\to}}
\newcommand{\onlong}[1]{\stackrel{#1}{\longrightarrow}}
\newcommand{\ol}[1]{\overline{#1}}
\newcommand{\wt}[1]{\widetilde{#1}}
\newcommand{\wh}[1]{\widehat{#1}}
\newcommand{\Cf}{\textit{cf.}\;}

\def\sn{\smallskip\noindent}

\def\ssm{\smallsetminus}

\newcommand{\Gm}{\mathbb{G}_{m}}
\newcommand{\CH}{\operatorname{CH}}
\newcommand{\Spec}{\operatorname{Spec}}

\newcommand{\dP}{\operatorname{\partial}_P}
\newcommand{\F}{\mathbb{F}}
\newcommand{\geo}{\mathrm{geo}}
\newcommand{\tor}{\mathrm{tor}} 
\newcommand{\kP}{k(P)}
\newcommand{\Kt}{K^{\times}}
\newcommand{\kt}{k^{\times}}

\newcommand{\mk}{\m_k}
\newcommand{\mup}{\mu_{p}}

\newcommand{\Ok}{\O_k}

\newcommand{\Ubar}{\ol{U}}
\newcommand{\Ubark}{\Ubar_k}
\newcommand{\UbarK}{\Ubar_K}

\renewcommand{\O}{\mathcal{O}}

\newcommand{\otimesZ}{\otimes_{\Z}}
\newcommand{\otimesM}{\stackrel{M}{\otimes}}

\newcommand{\Res}{\mathrm{Res}}

\newcommand{\ktp}{\kt/p}
\newcommand{\Ktp}{\Kt/p}

\title{Milnor $K$-groups attached to elliptic curves over a $p$-adic field}

\author{Toshiro Hiranouchi}

\begin{document}
\maketitle

\begin{abstract}
We study the Galois symbol map 
of the Milnor $K$-group attached to 
elliptic curves over a $p$-adic field. 
As by-products, 
we determine 
the structure of 
the Chow group  
for the product of elliptic curves 
over a $p$-adic field under some assumptions. 

\sn

\sn
\end{abstract}

\section{Introduction}
K.~Kato and M.~Somekawa 
introduced in \cite{Som90} the Milnor type $K$-group 
$K(k;G_1,\ldots ,G_q)$ 
attached to semi-abelian varieties $G_1,\ldots ,G_q$ 
over a field $k$ 
which is now called the Somekawa $K$-group.  
The group is defined by the quotient 
\begin{equation}
\label{eq:Som} 
  K(k;G_1,\ldots,G_q) := 
\left(\bigoplus_{k'/k:\,\mathrm{finite}} G_1(k') \otimesZ \cdots \otimesZ G_q(k')\right)/R
\end{equation}
where $k'$ runs through all finite extensions over $k$ 
and $R$ is the subgroup 
which produces ``the projection formula'' and 
``the Weil reciprocity law'' 
as in the Milnor $K$-theory 
(Def.\ \ref{def:Somekawa}). 
As a special case, 	
for the multiplicative groups 
$G_1 = \cdots = G_q = \Gm$, 
the group $K(k;\overbrace{\Gm,\ldots ,\Gm}^q)$ is 
isomorphic to $K_q^M(k)$
the ordinary Milnor $K$-group 
of the field $k$ (\cite{Som90}, Thm.\ 1.4). 
For any positive integer $m$ prime to the characteristic of $k$, 
let $G_i[m]$ be the Galois module defined by 
the kernel of $G_i(\kbar) \onto{m} G_i(\kbar)$ the multiplication by $m$. 
Somekawa defined also the Galois symbol map 
\[
  h :K(k;G_1,\ldots ,G_q)/m \to H^q(k,G_1[m]\otimes \cdots \otimes G_q[m])
\]
by the similar way as in the classical Galois symbol map 
$K_q^M(k)/m \to H^q(k, \mu_m^{\otimes q})$ 
on the Milnor $K$-group.  
He also presented a conjecture 
in which the map $h$ is injective 
for arbitrary field $k$. 
For the case $G_1 = \cdots = G_q = \Gm$, 
the conjecture holds by 
the Milnor-Bloch-Kato conjecture, 
now is a theorem of Voevodsky, Rost, and Weibel 
(\cite{Wei09}). 
However, it is also known that 
the above conjecture does not hold 
in general 
(see Conj.\ \ref{conj:SK} and its remarks below for the other known results).

The aim of this note is 
to show this conjecture 
for elliptic curves over a local field 
under some assumptions. 

%


\begin{theorem}[Thm.\ \ref{thm:main2}, Prop.\ \ref{prop:div}]
\label{thm:main}
$\mathrm{(i)}$  
Let  
$E_1,\ldots ,E_q$ be elliptic curves over $k$ 
with $E_i[p] \subset E_i(k)$ ($1\le i\le q$). 
Assume that $E_1$ is not a supersingular elliptic curve. 
Then, for $q\ge 3$,
\[
    K(k;E_1,\ldots , E_q)/p^n = 0
\]

\sn
$\mathrm{(ii)}$ 
Let  $E_1, E_2$ be elliptic curves over $k$ 
with $E_i[p^n] \subset E_i(k)$ $(i=1,2)$. 
Assume that $E_1$ is not supersingular. 
Then,  
the Galois symbol map 
\[
  h^2: K(k;E_1, E_2)/p^n \to H^2(k, E_1[p^n] \otimes E_2[p^n])
\]
is injective. 
\end{theorem}

The theorem above is known when 
$E_i$'s have 
semi-ordinary reduction 
(= good ordinary or multiplicative reduction)
(\cite{Yam05}, \cite{RS00}, see also \cite{MR09}). 
Hence our main interest is in 
supersingular elliptic curves.

In our previous paper \cite{HH13}, 
we investigate the image of the Galois symbol map $h^2$. 
As byproducts, 
we obtain the structure of 
the Chow group 
$\CH_0(E_1\times E_2)$ of $0$-cycles. 
By Corollary 2.4.1 in \cite{RS00}, 
we have 
\[
  \CH_0(E_1\times E_2) \simeq \Z \oplus E_1(k)\oplus E_2(k) \oplus K(k;E_{1}, E_{2}). 
\]
The Albanese kernel 
$T(E_1\times E_2) := \Ker(\alb: \CH_0(E_1\times E_2)^0 \to (E_1\times E_2)(k))$ 
coincides with the Somekawa $K$-group $K(k;E_1,E_2)$, 
where $\CH_0(E_1 \times E_2)^0$ 
is the kernel of 
the degree map $\CH_0(E_1\times E_2) \to \Z$.  
Mattuck's theorem \cite{Mat55} 
implies the following: 

\begin{corollary}
\label{thm:Chow}
Let $E_1$ and $E_2$ be elliptic curves over $k$ with 
good or split multiplicative reduction. 
Assume that $E_1$ is not a 
supersingular elliptic curve   
and $E_i[p^n] \subset E_i(k)$. 
Then, we have 
\[ 
\CH_0(E_1\times E_2)/p^n  \simeq 
 \begin{cases}
  (\Z/p^n)^{2[k:\Qp]+ 6},& \mbox{if $E_1$ and $E_2$ have a same reduction type},\\
  (\Z/p^n)^{2[k:\Qp] +7 },& 
\mbox{otherwise}.
  \end{cases}
\]
\end{corollary}

\sn
{\it Outline of this note.}\   
In Section \ref{sec:Somekawa}, 
we recall the definition and some properties 
of Somekawa $K$-group $K(k;G_1,\ldots ,G_q)$
attached to semi-abelian varieties $G_1,\ldots , G_q$ 
over a perfect field $k$. 
We also introduce the Mackey product 
$G_1\otimesM \cdots \otimesM G_q$  
which is defined as in (\ref{eq:Som})
but by factoring out a relation 
concerning ``the projection formula'' only. 
In Section \ref{sec:unit} we study 
the structure of the Mackey product $\Ubar^m \otimesM \Ubar^n$ 
over a $p$-adic field $k$. 
Here, $\Ubar^m$ is the Mackey functor 
defined by the higher unit groups of finite extensions over $k$ 
as a sub Mackey functor of 
the cokernel $\Gm/p$ of the multiplication by $p$ on $\Gm$. 
Tate \cite{Tat76}, Raskind and Spie\ss\ \cite{RS00} 
show that 
the Galois symbol map 
induces bijections 
\[
h^2:\left(\Gm/p \otimesM \Gm/p\right) (k) \isomto K_2^M(k)/p \stackrel{h}{\isomto} H^2(k,\mup^{\otimes 2}).
\]
We further calculate the kernel and the image of the composition 
\[
  h^{m,n}: \left(\Ubar^m\otimesM \Ubar^n\right) (k) \to \left(\Gm/p\otimesM \Gm/p\right) (k) \onto{h^2} H^2(k,\mup^{\otimes 2})  
\]
and determine the structure of $\Ubar^m \otimesM \Ubar^n$ partially. 
The proof of Theorem \ref{thm:main}   
is given in Section \ref{sec:local}. 

Throughout this note, 
for an abelian group $A$ and a non-zero integer $m$, 
let $A[m]$ be the kernel and $A/m$ the cokernel 
of the map $m:A\to A$ defined by multiplication by $m$. 
For a field $k$, 
we denote by
$G_k := \Gal(\ol{k}/k)$ the absolute Galois group of $k$ 
and 
$H^i(k, M) := H^i(G_k, M)$ the Galois cohomology group of $G_k$ 
for some $G_k$-module $M$. 
The tensor product $\otimes$ means $\otimesZ$. 

\subsection*{Acknowledgements}
A part of this note was written 
during a stay of the author at the Duisburg-Essen university. 
He thanks the institute for its hospitality. 
This work was supported by KAKENHI 30532551.

\section{Somekawa $K$-groups}
\label{sec:Somekawa}

Throughout this section, $k$ is a perfect field. 

\begin{definition}
\label{def:Mack}
A {\it Mackey functor} $A$ over $k$ 
is a contravariant 
functor from the category of \'etale schemes over $k$ 
to the category of abelian groups 
equipped with a covariant structure 
for finite morphisms 
such that 
$A(X_1 \sqcup X_2)  = A(X_1) \oplus A(X_2)$ 
and if 
\[
\xymatrix@C=15mm{
  X' \ar[d]_{f'}\ar[r]^{g'} & X \ar[d]^{f} \\
  Y' \ar[r]^{g} & Y
}
\]
is a Cartesian diagram, then the induced diagram 
\[
\xymatrix@C=15mm{
  A(X') \ar[r]^{{g'}_{\ast}} & A(X)\\
  A(Y') \ar[u]^{{f'}^{\ast}} \ar[r]^{{g}_{\ast}} & A(Y)\ar[u]_{f^{\ast}}
}
\]
commutes. 
\end{definition}

For a Mackey functor $A$, 
we denote by $A(K)$ 
its value $A(\Spec(K))$  
for a field extension $K$ over $k$.

\begin{definition}
\label{def:Mackey}
For Mackey functors $A_1,\ldots ,A_q$, 
their \textit{Mackey product} 
$A_1\otimesM \cdots \otimesM A_q$ 
is defined as follows: 
For any finite field extension $K/k$, 
\[
K \mapsto 
\left(A_1\otimesM \cdots \otimesM A_q \right) (K) 
 := \left(\bigoplus_{L/K:\, \mathrm{finite}} A_1(L) \otimes \cdots \otimes A_q(L)\right) /R, 
\]
where $R$ is the subgroup generated 
by elements of the following form: 

\sn
(PF) 
For any finite field extensions 
$K \subset K_1 \subset K_2$, and 
if $x_{i_0} \in A_{i_0}(K_2)$ and $x_i \in A_i(K_1)$ 
for all $i\neq i_0$, then 
\[
  j^{\ast}(x_1) \otimes \cdots \otimes x_{i_0} \otimes \cdots \otimes j^{\ast}(x_q) - x_1\otimes \cdots \otimes j_{\ast}(x_{i_0})\otimes \cdots \otimes x_q,
\]
where $j:\Spec(K_2) \to \Spec(K_1)$ is the canonical map. 

\end{definition}

This product gives a tensor product in the abelian category 
of Mackey functors with unit $\Z:k' \mapsto \Z$. 
We write $\{x_1,\ldots,x_q\}_{K/k}$ 
for the image of 
$x_1 \otimes \cdots \otimes x_q \in 
A_1(K) \otimes \cdots \otimes A_q(K)$ in the product 
$\left(A_1\otimesM \cdots \otimesM A_q\right)(k)$. 
For any field extension $K/k$, 
the canonical map $j=j_{K/k}:k\inj K$ induces  
the pull-back 
\[
  \Res_{K/k} := j^{\ast}: \left(A_1\otimesM \cdots \otimesM A_q\right)(k) \longrightarrow 
\left(A_1\otimesM \cdots \otimesM A_q\right)(K)
\] 
which is called the \textit{restriction map}. 
If the extension $K/k$ is finite, 
then the push-forward 
\[
  N_{K/k} := j_{\ast}: \left(A_1\otimesM \cdots \otimesM A_q\right)(K) \longrightarrow 
\left(A_1\otimesM \cdots \otimesM A_q\right) (k)
\] 
is given by 
$N_{K/k}(\{x_1,\ldots ,x_q\}_{L/K}) = \{x_1,\ldots, x_q\}_{L/k}$ 
on symbols and is called the \textit{norm map}. 
%
%
A smooth commutative algebraic group $G$ over $k$ 
forms a Mackey functor defined by $K \mapsto G(K)$.  
For a field extension $K_2/K_1$, 
the pull-back is the canonical map
given by $j:K_1 \inj K_2$ 
which is denoted by $j^{\ast}= \Res_{K_2/K_1}:G(K_1)\inj G(K_2)$. 
When the extension $K_2/K_1$ is finite, 
the push-forward is written as 
$j_{\ast} = N_{K_2/K_1}:G(K_2)\to G(K_1)$.



\begin{definition}
\label{def:Somekawa}
Let $G_1,\ldots , G_q$ be semi-abelian varieties over $k$. 
The \textit{Somekawa $K$-group} attached to $G_1,\ldots ,G_q$ 
is defined by 
\[
  K(k;G_1,\ldots ,G_q) := \left(G_1 \otimesM \cdots \otimesM G_q \right) (k)/R,
\]
where $R$ is the subgroup 
generated by the elements of the following form: 
 
\noindent
(WR) Let $k(C)$ be the function field of 
a projective smooth curve $C$ over $k$. 
For $g_i \in G_i(k(C))$ and $f\in k(C)^{\times}$, 
assume that for each closed point $P$ in $C$ there exists 
$i(P)$ ($1\le i(P) \le q$) 
such that $g_i \in G_i(\O_{C,P})$ for all $i\neq i(P)$. 
Then 
\[ 
  \sum_{P \in C_0}g_1(P)\otimes \cdots \otimes \dP(g_{i(P)}, f)\otimes \cdots \otimes g_{q}(P) \in R.
\]
Here $C_0$ is the set of closed points in $C$, 
$g_i(P) \in G_i(k(P))$ 
denotes the image of $g_i$ under the canonical map $G_i(\O_{C,P}) \to G_i(\kP)$  
 and 
$\dP:G_i(k(C)) \times {k(C)}^{\times} \to G_i(\kP)$ 
is the local symbol (\cite{Som90}). 
%
\end{definition}

%

For an isogeny $\phi:G \to H$ of semi-abelian varieties, 
the exact sequence $0 \to G[\phi] \to G(\kbar) \onto{\phi} H(\kbar)\to 0$ 
of Galois modules 
gives an injection of Mackey functors 
\[
  h^1: H/\phi   \to H^1(-, G[\phi]),
\]
where $H/\phi := \Coker(\phi)$ (in the category of Mackey functors) 
and $H^1(-,G[\phi])$ is also the Mackey functor given by 
$K \mapsto H^1(K, G[\phi])$. 
For isogenies $\phi_i :G_i\to H_i$ ($1\le i\le q$), 
the cup products and the norm map (=the corestriction) 
on the Galois cohomology groups give 
\begin{equation}
\label{Galois}
  h^q: \left(H_1/\phi_1 \otimesM \cdots \otimesM H_q/\phi_q\right) 
\xrightarrow{N_{-/k}\circ (h^1(-) \cup h^1(-) \cup \cdots  \cup h^1(-) )} 
H^q(-, G_1[\phi_1]\otimes \cdots \otimes G_q[\phi_q]).
\end{equation}
For any positive integer $m$ prime to the characteristic of $k$,  
we consider an isogeny
$m:G_i\to G_i$ induced from the multiplication by $m$. 
The Galois symbol map 
$h^q: G_1/m \otimesM \cdots \otimesM G_q/m  \to H^q(k,G_1[m]\otimes\cdots \otimes G_q[m])$ 
(\ref{Galois}) 
factors through 
$K(k;G_1, \ldots ,G_q)/m$  (\cite{Som90}, Prop.\ 1.5) 
and the induced homomorphism 
\[
  h^q_m :K(k;G_1,\ldots , G_q)/m \to H^q(k,G_1[m]\otimes \cdots \otimes G_q[m])
\]
is called the \textit{Galois symbol map}. 

\begin{conjecture}[Kato-Somekawa, \cite{Som90}]
  \label{conj:SK}
  Let $G_1,\ldots ,G_q$ be semi-abelian varieties over $k$. 
  For any positive integer $m$ 
  prime to the characteristic of $k$, 
  the Galois symbol map $h^q_m$ is injective.
\end{conjecture}

The surjectivity of the Galois symbol map 
does not hold in general 
(For example, see (\ref{eq:image}) in Sect.\ \ref{sec:local}).
The above conjecture is studied in the following special semi-abelian varieties:

\begin{itemize}
  \setlength{\parskip}{0cm} 
  \setlength{\itemsep}{0cm} 
\item[(a)] \textit{Case where $G_1 = \cdots = G_q = \Gm$}:\ 
The conjecture and more strongly 
the bijection of the Galois symbol map are known for 
the multiplicative groups $G_1 = \cdots = G_q = \Gm$.  
by the Bloch-Kato conjecture (a theorem of Voevodsky, Rost and Weibel \cite{Wei09}). 
In fact, the group  
$K(k;\overbrace{\Gm,\ldots ,\Gm}^{q})$ 
coincides with the Milnor $K$-group 
$K_q^M(k)$ (\cite{Som90}, Thm.\ 1.4) 
and the map $h^q_m$ is the ordinary Galois symbol map. 

\item[(b)] \textit{Case where $G_1$ and $G_2$ are tori}: 
Yamazaki proved this conjecture 
for $G_1 = T$ a torus which admits a ``motivic interpretations'' 
(e.g., $k$ is a non-archimedean local field, \cite{Yam09}, Thm.\ 3.2) 
and $G_2 = \Gm$ 
(\cite{Yam09}, Prop.\ 2.11) 
and disproved it for general tori with M.~Spie\ss\ (\cite{SY09}, Prop.\ 7). 
Hence 
the above conjecture does not hold in general. 

\item [(c)] \textit{Case where $G_1 = \Gm$ and $G_2$ is a Jacobian variety}:\ 
It is known also (by Spie\ss, \cite{Yam05}, Appendix)
the conjecture holds 
for $G_1 = \Gm$ and $G_2 = J_X$ the Jacobian variety 
of a smooth projective geometrically connected curve $X$ over $k$ 
with $X(k)\neq \emptyset$. 

\item [(d)] \textit{Case where $k$ is a finite field}:\ 
If $k$ is a finite field, 
we have $K(k;G_1,\ldots , G_q) = 0$ for $q\ge 2$ 
(Kahn, \cite{Kahn92}). 
The conjecture holds in this case. 

\item [(e)] \textit{Case where $k$ is a $p$-adic field}:\ 
Let $A_1,\ldots, A_q$ be abelian varieties with split semi-ordinary reduction over $k$. 
If we assume that $A_1[m],\cdots ,A_q[m]$ are $k$-rational, 
then the conjecture holds (\cite{RS00}, Rem.\ 4.5.8 (b), see also \cite{Yam05}, Thm.\ 4.3). 
Note also that 
$K(k;A_1,\ldots ,A_q)$ is divisible for $q\ge 3$ (\cite{RS00}, Rem.\ 4.4.5). 

\end{itemize}

%
%
%

\section{Higher unit groups}
\label{sec:unit}
Let $k$ be a finite field extension  of $\Qp$ 
assuming $p\neq 2$. 
We denote by 
$v_k$ the normalized valuation, 
$\m_k$    the maximal ideal of the valuation ring $\O_k$, 
$\O_k^{\times} = U_k^0$  the group of units in $\O_k$ 
and $\F = \O_k/\m_k$ the (finite) residue field.  
In this section we study the Mackey product 
of the Mackey functors $\Ubar^m$ defined by 
the higher unit groups. 
The higher unit groups $U_k^m := 1 + \m_k^m$ 
induce the filtration $(\Ubar^m_k)_{m\ge 0}$ 
in $\ktp$ 
which is  
given by 
$\Ubark^m := \Im(U^m_k \to \ktp)$. 
The structure of the graded pieces of this filtration is known as follows. 

\begin{lemma}[\Cf \cite{Kaw02}, Lem.\ 2.1.3]
\label{lem:str}
Put $e_0 := e_0(k):= v_k(p)/(p-1)$. 
Assume $\mup := \Gm[p] \subset k$. 

\sn
$\mathrm{(a)}$ If $0 \le m < pe_0$, then 
\[
  \Ubark^m/\Ubark^{m+1} \simeq 
  \begin{cases}
  \F,& \mathrm{if}\ p\nmid m,\\
  1, & \mathrm{if}\ p\mid m.
  \end{cases}
\]

\sn
$\mathrm{(b)}$ If $m = pe_0$, then 
$\Ubark^{pe_0}/\Ubark^{pe_0+1} \simeq \Z/p$.

\sn
$\mathrm{(c)}$ If $m>pe_0$, then $\Ubark^m = 1$.
\end{lemma}

We define a sub Mackey functor 
$\Ubar^m$ of $\Gm/p :=\Coker(p:\Gm\to \Gm)$ by 
$K \mapsto \Ubar_{K}^{me_{K/k}}$ 
for any positive integer $m$, 
where $e_{K/k}$ is the ramification index of 
the extension $K/k$. 
\begin{lemma}[\cite{RS00}, Lem.\ 4.2.1]
\label{lem:4.2.1}
For $i,j\ge 0$ with $i+j \ge 2$, the Galois symbol map induces 
\[
  (\Ubar^0)^{\otimesM i} \otimesM (\Gm/p)^{\otimesM j} \isomto 
  \begin{cases}
    H^2(-, \mu_p^{\otimes 2}),& \text{if $i+j = 2$}, \\
    0,& \text{otherwise}.
  \end{cases}
\]
\end{lemma}
For each positive integers $m$ and $n$, 
now we define a map $h^{m,n}$
by the composition
\begin{align*}
\label{eq:hmn}
  h^{m,n}:\Ubar^m\otimesM \Ubar^n \to \Gm/p\otimesM \Gm/p\onto{h^2} H^2(-,\mup^{\otimes 2}). 
\end{align*}
Here, the latter $h^2$ is the Galois symbol map on $\Gm/p\otimesM \Gm/p$  
and is bijective (Lem.\ \ref{lem:4.2.1}). 
We also denote by 
\[
h^{-1,m}: \Gm/p \otimesM \Ubar^m \to \Gm/p\otimesM \Gm/p \onto{h^2} H^2(-,\mup^{\otimes 2})
\]
by convention. 
We determine the structure of the Mackey 
product of these Mackey functors $\Ubar^m$ as follows.

\begin{lemma}
\label{lem:key}
Let $k$ be a $p$-adic field which contains 
$\mup$. 
Put $e_0 := e_0(k) := v_k(p)/(p-1)$. 

\sn
$\mathrm{(i)}$ 
For any positive integer $m$, 
the map $h^{-1,m}$ induces an isomorphism 
\[
  \Gm/p\otimesM \Ubar^m \isomto 
  \begin{cases}
  H^2(-,\mu_p^{\otimes 2}), & \mbox{if}\ m\le pe_0,\\
  0, & \mbox{otherwise}. 
  \end{cases}    
\]

\sn
$\mathrm{(ii)}$  
For any positive integer $m$, 
the map $h^{0,m}$ induces an isomorphism 
\[
\Ubar^{0} \otimesM \Ubar^{m} 
\isomto  
  \begin{cases}
  H^2(-,\mu_p^{\otimes 2}), & \mbox{if}\ m < pe_0,\\
  0, & \mbox{otherwise}. 
  \end{cases}    
\]

\sn
$\mathrm{(iii)}$ 
The map $h^{pe_0,m}$ induces 
\[
  \Ubar^{pe_0} \otimesM \Ubar^{m} \isomto 0.
\] 
\end{lemma}

The rest of this section is devoted to show this lemma. 
For any finite extension $K/k$, 
the cup product 
$\cup: H^1(K,\mup)\otimes H^1(K,\mup) \to H^2(K, \mup^{\otimes 2})$ 
on the Galois cohomology groups 
is characterized by the Hilbert symbol 
$(\ ,\ )_p :\Ktp \otimes \Ktp \to \mup$ 
as in the following commutative diagram 
(\Cf \cite{Ser68}, Chap.\ XIV):
\begin{equation}
\label{eq:Hilbert}
\vcenter{
  \xymatrix{
     H^1(K ,\mup)\otimes H^1(K, \mup) \ar[r]^-{\cup} & H^2(K, \mup^{\otimes 2})\ar[d]^{\simeq} \\
    \Ktp\otimes \Ktp \ar[u]^{\simeq} \ar[r]^-{(\ ,\ )_p} & \mup&.
  }
}
\end{equation}
The order of the image in $H^2(K,\mup^{\otimes 2}) \simeq \mup \simeq \Z/p$ 
by the Hilbert symbol are calculated by local class field theory (\Cf \cite{HH13}, Lem.\ 3.1): 

\begin{lemma}
\label{lem:Hilbert}
Let $m$ and $n$ be positive integers.

\sn 
$\mathrm{(i)}$     
\[
  \# (\Ktp, \UbarK^m)_p = 
  \begin{cases}
    p, & \mathit{if}\ m \le pe_0(K),\\
    0, & \mathit{otherwise}.
  \end{cases}
\]

\sn 
$\mathrm{(ii)}$ If $p\nmid m$ or $p\nmid n$, then 
\[
  \# (\UbarK^m, \UbarK^n)_p = 
  \begin{cases}
    p, & \mathit{if}\ m+n \le pe_0(K),\\
    0, & \mathit{otherwise}.
  \end{cases}
\]

\sn
$\mathrm{(iii)}$ If $p\mid m$ and $p\mid n$, 
then 
\[
  \# (\UbarK^m, \UbarK^n)_p = 
  \begin{cases}
    p, & \mathit{if}\ m+n < pe_0(K),\\
    0, & \mathit{otherwise}.
  \end{cases}
\]
\end{lemma}

From the lemma above, 
the image of the product does not depend on an 
extension $K/k$, 
hence we obtain the images as required in Lemma \ref{lem:key}.

\begin{lemma}
\label{lem:Step1}
For any symbol $\{a,b\}_{K/K}$ 
in $\left(\Ubar^0\otimesM \Ubar^m\right)(K)$, 
if we assume $h^{0,m}(\{a,b\}_{K/K}) = 0$ then  $\{a,b\}_{K/K} = 0$. 
\end{lemma}
\begin{proof}
The symbol map is written by the Hilbert symbol 
$h^{0,m}(\{a,b\}_{K/K}) = (a, b)_p$ as in (\ref{eq:Hilbert}) and thus 
$a$ is in the image of the norm 
$N_{L/K}:\Ubar_{L}^0 \to \UbarK^0$ 
for $L = K(\sqrt[p]{b})$ 
(\cite{FV02}, Chap.\ IV, Prop.\ 5.1). 
Take  $\wt{a} \in \Ubar_L^0$ 
such that $N_{L/K}(\wt{a}) = a$. 
We obtain
\[
  \{a,b\}_{K/K} = \{N_{L/K}\wt{a}, b\}_{K/K}  
 = \{\wt{a}, \Res_{L/K}(b)\}_{L/K} = 0 
\]  
by the condition (PF) 
in the definition of the Mackey product  
(Def.\ \ref{def:Mackey}).  
\end{proof}

\begin{proof}[Proof of Lem.\ \ref{lem:key}]
(iii) For any symbol $\{a,b\}_{L/K}$ in $\left(\Ubar^{pe_0} \otimesM \Ubar^m \right) (K)$, 
we have $N_{L/K}(\{a,b\}_{L/L}) = \{a,b\}_{L/K}$. 
Thus it is enough to show $\{a,b\}_{K/K}  = 0$ 
with $a\neq 1$. 
Put $e = e_{K/k}$. 
Since the extension $L=K(\sqrt[p]{a})$ is unramified of degree $p$ 
(\cite{Kaw02}, Lem.\ 2.1.5), 
the norm map $N_{L/K}:\Ubar_{L}^{me}\to \Ubar_{K}^{me}$ is surjective 
(\cite{Ser68}, Chap.\ V, Sect.\ 2, Prop.\ 3). 
By the projection formula (PF), 
\[
  \{a,b\}_{K/K} = \{a, N_{L/K}\wt{b}\}_{K/K}  
 = \{\Res_{L/K}(a), \wt{b}\}_{L/L} = 0 
\]  
for some $\wt{b}\in \Ubar_{L}^{me}$.

\sn
(i) The assertion (i) is proved by similar arguments as in (ii) below.

\sn
(ii) 
For $m\ge pe_0$, 
the assertion follows from Lemma \ref{lem:Hilbert} and Lemma \ref{lem:Step1}. 
Assume $m<pe_0$. 
For any finite extension $K/k$ with ramification index $e$, 
we denote by $S(K)$ the subgroup of $\left(\Ubar^0\otimesM \Ubar^{m}\right)(K)$ 
generated by symbols of the form 
$\{a,b\}_{K/K}$ 
($a\in \UbarK^0 = \UbarK^1, b\in \Ubar^{m}(K) = \UbarK^{me}$). 
Put 
\[
  n = 
  \begin{cases}
  me+1,& \mbox{if $p\mid me$},\\
  me,& \mbox{if $p\nmid me$}. 
  \end{cases}
\]
We also denote by $T(K) \subset S(K)$ 
the subgroup generated by $\{a,b\}_{K/K}$ 
with $a\in \Ubar^{pe_0(K)-n}_K$. 
Fix a uniformizer $\pi$ of $K$ and 
let $\F_K= \O_K/\pi\O_K$ be the residue field of $K$.  
Define  
$\phi:\F_K \to T(K)$ 
by 
$x \mapsto \{1+ \wt{x}\pi^{pe_0(K)-n},  1 + \pi^{n}\}_{K/K}$, 
where $\wt{x} \in \O_{K}$ is a lift of $x$. 
Lemma \ref{lem:Hilbert} and Lemma \ref{lem:Step1} 
imply $\{1+ a\pi^{pe_0(K)-n+1},  1 + \pi^{n}\}_{K/K} = 0$ 
for $a\in \O_K$. 
Thus the map $\phi$ does not depend on the choice of $\wt{x}$. 
By calculations of symbols (\Cf \cite{BK86}, Lem.\ 4.1), 
we have 
\begin{align*}
  h^{0,m}(\phi(x)) 
  & = (1 + \wt{x}\pi^{pe_0(K)-n}, 1+ \pi^{n})_p \\
  &= (1 + \wt{x}\pi^{pe_0(K)-n}, 1+ (1 + \wt{x}\pi^{pe_0(K)-n}) \pi^{n})_p\\
  &= -(1 + \wt{x}\pi^{pe_0(K)}, -\wt{x} \pi^{n})_p\\
  &= -n(1+\wt{x}\pi^{pe_0(K)}, \pi)_p.
\end{align*}
The map $\phi$ is non-zero, 
since its image by $h^{0,m}$ 
generates $H^2(K,\mup^{\otimes 2}) \simeq \Z/p$. (\cite{Milnor}, Cor.\ A.12, see also \cite{BK86}).
Define the group homomorphism 
$\sigma:\F_K\to \F_K$ by $x\mapsto x^p + ax$, 
where $a$ is the class in $\F_K$ represented by $p\pi^{-v_K(p)}$. 
From the equality (\Cf \cite{BK86}, Lem.\ 5.1)
\[
  (1+\wt{x} \pi^{e_0(K)})^p \equiv 1 + (\wt{x}^p + \wt{x}p\pi^{-v_K(p)}\pi^{pe_0(K)} \mod \pi^{pe_0(K) +1}, 
\]
we obtain
\begin{align*}
  h^{0,m}(\phi(x^p + ax)) 
  &= -n(1 + (\wt{x}^p+ p\pi^{-v_K(p)}\wt{x})\pi^{pe_0(K)}, \pi)_p\\
  &= -n( (1 + \wt{x}\pi^{e_0(K)})^p, \pi^{n})_p\\
  & = 0. 
\end{align*}
By Lemma \ref{lem:Step1},  
the map $\phi$ factors through $\F_K/\sigma(\F_K)$. 
On the other hand, 
the map $\sigma$ is extended to $\sigma:\ol{\F}_K\to \ol{\F}_K$ and we have $H^1(\F_K, \ol{\F}_K) = 1$. 
Since $\Ker(\sigma) \simeq \Z/p$ as Galois modules, 
we obtain 
$\F_K/\sigma(\F_K) \simeq H^1(\F_K, \Ker(\sigma)) \simeq \Z/p$. 
Now we have the following commutative diagram 
\[
  \xymatrix@C=15mm{
    \F_K/ \sigma (\F_K) \ar[rd]_{\simeq} \ar@{^{(}->}[r]^{\phi} & T(K)\ar@{->>}[d]^{h^{0,m}} \\
          & H^2(K,\mup^{\otimes 2}).
  }
\]
Therefore, $T(K) = \Im(\phi) \oplus \Ker(h^{0,m})$. 
We show that for any element $x = \sum_{i=1}^n \{a_i,b_i\}_{K/K} \in T(K)$, 
if $h^{0,m}(x) = 0$, then $x = 0$ by induction on $n$. 
For any symbol $\{a,b\}_{K/K}$ in $T(K)$ with $h^{0,m}(\{a,b\}_{K/K}) = 0$. 
Lemma \ref{lem:Step1} implies  $\{a,b\}_{K/K} = 0$.\ 
Take an element $x = \sum_{i=1}^n \{a_i,b_i\}_{K/K}$ 
with $h^{0,m}(x) = 0$. 
If $\{a_i,b_i\}_{K/K} \in \Im(\phi)$ for all $i$, 
then $x = \sum_{i=1}^n\{a_i,b_i\}_{K/K}$ is in $\Im(\phi)$. 
Hence $h^{0,m}(x) = 0$ implies $x = 0$ from the above diagram. 
When there exists $j$ such that $\{a_j,b_j\}_{K/K} \not\in \Im(\phi)$, 
we have $h^{0,m}(\{a_j,b_j\}_{K/K}) = 0$ 
and thus $\{a_j,b_j\}_{K/K} = 0$. 
By the induction hypothesis, 
$h^{0,m}(x) = h(\sum_{i\neq j}\{a_i,b_i\}_{K/K}) = 0$ implies 
$0 = \sum_{i\neq j}\{a_i,b_i\}_{K/K} = x$.  
By the same manner with the following commutative diagram, 
\[
  \xymatrix@C=15mm{
    T(K) \ar[rd]_{\simeq} \ar@{^{(}->}[r] & S(K) \ar@{->>}[d]^{h^{0,m}} \\
          & H^2(K,\mup^{\otimes 2}).
  }
\]
the map $h^{0,m}$ is injective on $S(K)$ and thus $S(K) \simeq \Z/p$. 

Next, we show 
$S(K) = \left(\Ubar^0\otimesM \Ubar^m\right)(K)$. 
Take a symbol $\{a,b\}_{L/K}\neq 0$ and prove that 
it is in $S(K)$ by induction on the exponent of $p$ in the ramification index $e_{L/K}$ of $L/K$. 

\sn
\textit{The case $p\nmid e_{L/K}$}. 
In this extension $L/K$, 
the norm map $N_{L/K}:\Ubar_{L}^0 \to \Ubar_K^0$ 
is surjective. 
There exist $\wt{c} \in \Ubar^0(L)$ and $d \in \Ubar^{m}(K) = \Ubar_K^{me}$ 
such that $\{N_{L/K}(\wt{c}), d\}_{K/K}$ 
is a generator of $S(k) \simeq \Z/p$. 
By the projection formula, we have
\[
  \{N_{L/K}(\wt{c}), d\}_{K/K} = \{\wt{c}, \Res_{L/K}(d)\}_{L/K} 
  = N_{L/K}(\{\wt{c}, \Res_{L/K}(d)\}_{L/L}). 
\]
Because of 
the symbol $\{\wt{c} , \Res_{L/K}(d)\}_{L/L}$ is also a generator 
of $S(L)$, 
for some $i$ we obtain
\begin{align*}
  \{a,b\}_{L/K} &= N_{L/K}(\{\wt{c}^i, \Res_{L/K}(d)\}_{L/L})\\
                &= \{N_{L/K}(\wt{c}^i), d\}_{K/K}.
\end{align*}

\sn
\textit{The case $p\mid e_{L/K}$}. 
There exists a finite extension $M/L$ of degree 
prime to $p$ and an intermediate field $M_1$ of $M/K$ 
such that $M/M_1$ is a cyclic and totally ramified extension of 
degree $p$. 
The norm map $N_{M/L}:U_M^0/p\to U_{L}^0/p$ 
is surjective and we have
 $\{a,b\}_{L/K} = \{N_{M/L}(\wt{a}), b\}_{L/K} = \{\wt{a}, \Res_{M/L}(b)\}_{M/K}$ 
for some $\wt{a} \in \Ubar_M^0$. 
There exists an element 
$c\in \Ubar^{m}(M_1) = \Ubar_{M_1}^{me_{M_1/k}}$ such that $\Sigma = M_1(\sqrt[p]{c})$ 
is a totally ramified nontrivial extension of $M_1$ 
and $\Sigma \neq M$. 
In fact if the element $c$ is 
in $\ol{U}_{M_1}^{i} \ssm (\ol{U}_{M_1}^{i+1})$  
($me_{M_1/k} < i < pe_0(M_1) = pe_0e_{M_1/k}, p\nmid i)$ 
then 
the upper ramification subgroups of $G := \Gal(\Sigma/M_1)$ 
(\cite{Ser68}, Chap.\ IV) 
is known to be  
\[
G = G^0 = G^1 = \cdots = G^{pe_0(M_1) - i} \supset G^{pe_0(M_1) - i+1} = 1
\] 
(\cite{Kaw02}, Lem.\ 2.1.5, 
see also \cite{Ser68}, Chap.\ V, Sect.\ 3). 
Hence we can choose $c$ such that the ramification break 
of $G$
is different from the one of $\Gal(M/M_1)$. 
Thus $N_{\Sigma/M_1}(U_{\Sigma}^0) + N_{M/M_1}(U_M^0) = U_{M_1}^0$ 
and we can take a symbol 
$\{N_{M/M_1}(\wt{d}), c\}_{M_1/M_1}$ such that 
it is a generator 
of $S(M_1)$ for some $\wt{d} \in \Ubar_{M}^0$ 
and thus $\{\wt{d}, \Res_{M/M_1}(c)\}_{M/M}$ 
is also a generator of $S(M)$. 
Therefore, for some $i$, we have 
\begin{align*}
  \{a,b\}_{L/K} &= \{\wt{a}, \Res_{M/L}(b)\}_{M/K}\\
                &= N_{M/K}\{\wt{a}, \Res_{M/L}(b)\}_{M/M}\\
                &= N_{M/K} \{\wt{d}^i, \Res_{M/M_1}(c)\}_{M/M}\\
                &= N_{M_1/K} \circ N_{M/M_1}\{\wt{d}^i, \Res_{M/M_1}(c)\}_{M/M}\\
                &= N_{M_1/K}\{\wt{d}^i, \Res_{M/M_1}(c)\}_{M/M_1}\\ 
                &= N_{M_1/K}\{N_{M/M_1}(\wt{d}^i), c\}_{M_1/M_1}\\ 
                &= \{N_{M/M_1}(\wt{d}^i), c\}_{M_1/K}. 
\end{align*}
From the induction hypothesis, 
the symbol $\{a,b\}_{L/K}$ is in $S(K)$. 
Therefore, we obtain $S(K) = \left(\Ubar^0 \otimesM \Ubar^m\right)(K)$. 
Hence $h^{0,m}:\Ubar^0 \otimesM \Ubar^m \to H^2(-,\mu_p^{\otimes 2})$ 
is an isomorphism and the assertion follows. 
\end{proof}

\section{Galois symbol map for elliptic curves}
\label{sec:local}
We keep the notation as in the last section: 
$k$ is a $p$-adic field  
assuming $p\neq 2$ with residue field $\F = \Ok/\mk$ 
and $e_0 = v_k(p)/(p-1)$. 
The main result here is the following theorem:
 
\begin{theorem}
\label{thm:main2} 
Let  $E_1, E_2$ be elliptic curves over $k$ 
with $E_i[p^n] \subset E_i(k)$ $(i=1,2)$. 
Assume that $E_1$ is not supersingular. 
Then 
the Galois symbol map 
\[
  h^2: K(k;E_1, E_2)/p^n \to H^2(k, E_1[p^n] \otimes E_2[p^n])
\]
is injective. 
\end{theorem}
\begin{proof}
Consider the following diagram with exact rows: 
\[
  \xymatrix{
    K(k;E_1,E_2)/p^{n-1} \ar[d]_{h_{p^{n-1}}^2}\ar[r] & K(k;E_1,E_2)/p^n\ar[d]_{h_{p^n}^2} \ar[r] & K(k;E_1,E_2)/p \ar[d]_{h_{p}^2} \\
    H^2(k,E_1[p^{n-1}]\otimes E_2[p^{n-1}]) \ar[r]& H^2(k,E_1[p^n]\otimes E_2[p^{n}]) \ar[r] & H^2(k,E_1[p] \otimes E_2[p]).
  }
\]
The assumption 
$E_i[p^n] \subset E_i(k)$ 
implies the injectivity of the left lower map  
$H^2(k,E_1[p^{n-1}]\otimes E_2[p^{n-1}]) \to H^2(k, E_1[p^n]\otimes E_2[p^{n}])$. 
By induction on $n$, 
the assertion follows from the case of $n=1$. 
By taking a finite field extension whose extension degree is prime to $p$, 
we may assume that $E_1$ and $E_2$ do not have additive reductions. 
The assertion follows from the following slightly stronger claim than the required. 
\end{proof}

\begin{theorem}
\label{thm:m=p}
Let  $E_1, E_2$ be elliptic curves over $k$ 
with $E_i[p] \subset E_i(k)$ $(i=1,2)$. 
Assume that $E_1$ is not supersingular. 
Then, 
\[
  h^2: \left(E_1\otimesM E_2\right)(k)/p \to H^2(k,E_1[p]\otimes E_2[p])
\]
is injective. 
\end{theorem}

We recall the following results on the image 
of the Kummer map $h^1:E(k) \to H^1(k,E[p])$ 
for an elliptic curve $E$ over $k$. 
(\cite{Kaw02}, see also \cite{Tak11}, Rem.\ 3.2). 
Assume $E[p] \subset E(k)$ and  
fix an isomorphism of the Galois modules  
$E[p] \simeq (\mup)^{\oplus 2}$. 
From the isomorphism, 
we can identify $H^1(k,E[p])$ and $(\ktp)^{\oplus 2}$.  
On the latter group $\kt/p$, 
the higher unit groups $U^m_k = 1 + \m_k^m$ 
induce a filtration 
$\Ubark^m := \Im(U^m_k \to \ktp)$ 
as noted in the last section. 
In terms of this filtration, 
the image of 
$h^1_{E}:E(k)/p \inj H^1(k,E[p]) = (\ktp)^{\oplus 2}$  
is written precisely as follows (\Cf \cite{Tak11}): 
\begin{equation}
\label{eq:image}
  \Im(h^1_E) = \begin{cases}
       \Ubark^0 \oplus \Ubark^{pe_0} & \mbox{if $E$: ordinary,}\\
       \Ubark^{p(e_0 - t_0)} \oplus \Ubark^{pt_0} & 
\mbox{if $E$: supersingular},\\
\ktp \oplus 1 & \mbox{if $E$: (split) multiplicative}.
\end{cases}
\end{equation}
Here the invariant $t_0 := t_0(E)$ 
is defined 
by 
\begin{equation}
\label{eq:t_0}
t_0(E) = \max\set{i | P \in \wh{E}(\mk^i)\ \mathrm{for\ all}\ P \in \wh{E}[p]}
\end{equation}
where $\wh{E}$ is the formal group associated to $E$. 
Note also the invariant $t_0$ satisfies $0 < t_0 < e_0$ 
and is calculated from the theory of the canonical subgroup 
of Katz-Lubin (\Cf \cite{HH13}, Thm.\ 3.5).

\begin{proof}[Proof of Thm.\ \ref{thm:m=p}]
Fix isomorphisms  of Galois modules  
$E_1[p] \simeq \mup^{\oplus 2}$ and $E_2[p] \simeq \mup^{\oplus 2}$. 
From the isomorphism  we can identify 
$H^1(-,E_1[p]) \simeq (\Gm/p)^{\oplus 2}$ 
and 
$H^1(-,E_2[p]) \simeq (\Gm/p)^{\oplus 2}$.

\sn
(a) \textit{$E_1$ has split multiplicative reduction}: 
Consider the case that $E_1$ has split multiplicative reduction. 
We also assume that $E_2$ has supersingular good reduction. 
Other cases on $E_2$ 
are treated in the same way and much easier. 
From (\ref{eq:image}), 
the Kummer maps on $E_1$ and $E_2$ induces isomorphisms 
\[
  E_1/p \isomto \Gm/p ,\quad E_2/p \isomto \Ubar^{p(e_0 - t_0)} \oplus \Ubar^{pt_0},
\]
where $t_0 := t_0(E_2)$. 
Therefore 
$E_1/p\otimesM E_2/p \simeq (\Gm/p\otimesM \Ubar^{p(e_0-t_0)}) \oplus 
(\Gm/p \otimesM \Ubar^{pt_0})$. 
The Galois symbol map $h^2$ commutes with 
the map $h^{-1,p(e_0-t_0)}$ and $h^{-1, pt_0}$ 
defined in the last section and 
the injectivity of $h^2$ follows from Lemma \ref{lem:key} (i). 

\sn
(b) \textit{$E_1$ has ordinary good reduction}: 
Next we assume that $E_1$ has ordinary good reduction 
and $E_2$  is an supersingular elliptic curve over $k$. 
In this case also, 
by (\ref{eq:image}) we have 
\[
E_1/p \isomto \Ubar^{0} \oplus \Ubar^{pe_0}. 
\]
We have to show that the induced Galois symbol maps  
on 
\[
  \Ubar^0\otimesM \Ubar^{p(e_0-t_0)}, \Ubar^0\otimesM \Ubar^{pt_0}, \Ubar^{pe_0}\otimesM \Ubar^{p(e_0-t_0)},\ 
\mathrm{and}\ \Ubar^{pe_0} \otimesM \Ubar^{p(e_0-t_0)}
\]
are injective.
However, 
the latter two are trivial by Lemma \ref{lem:key} (iii). 
The rest of the assertions follow from Lemma \ref{lem:key} (ii). 
\end{proof}

\begin{proposition}
\label{prop:div}
Let  
$E_1,\ldots ,E_q$ be elliptic curves over $k$ 
with $E_i[p] \subset E_i(k)$ ($1\le i\le q$). 
Assume that $E_1$ is not a supersingular elliptic curve. 
Then for $q\ge 3$,
\[
\left(E_1\otimesM \cdots \otimesM E_q\right)(k)/p^n = K(k;E_1,\ldots , E_q)/p^n = 0
\]
\end{proposition}

\begin{proof}
It is enough to show $(E_1\otimesM E_2 \otimesM E_3)/p = 0$. 
We show only the case $E_1$ has ordinary reduction and 
$E_i$ has supersingular reduction for each $i = 2,3$. 
As in the above proof, we have
\[
  E_1/p \isomto \Ubar^0\oplus \Ubar^{pe_0} ,\quad 
  E_i/p \isomto \Ubar^{p(e_0 - t_0(E_i))} \oplus \Ubar^{pt_0(E_i)} \ (i = 2,3),
\]
By Lemma \ref{lem:key}, we have 
\[
\Ubar^0\otimesM \Ubar^{p(e_0 - t_0(E_2))} \simeq \Ubar^0\otimesM \Ubar^{pt_0(E_2)} \simeq \Gm \otimesM \Gm.
\]
Hence the assertion follows from Lemma \ref{lem:4.2.1}.
\end{proof}

\begin{remark}
From the same arguments in the proof of Theorem\ \ref{thm:main2}, 
we also obtain 
the injectivity of the Galois symbol map 
\[
  h^2: K(k;\Gm, E)/p^n \to H^2(k, \Gm[p^n] \otimes E[p^n])
\]
under the assumption $E[p^n] \subset E(k)$. 
As in \cite{HH13} we can determine 
the image of the above $h^2$ and have
\[
K(k;\Gm,E)/p^n \simeq \begin{cases}
  \Z/p^n,& \mbox{if $E$: multiplicative},\\
  (\Z/p^n)^{\oplus 2},& \mbox{if $E$: good reduction}.
\end{cases}
\]
It is known that the Somekawa $K$-group $K(k;\Gm, E)$ 
is isomorphic to 
the homology group $V(E)$ 
of the complex 
\[
  K_2(k(E)) \onlong{\oplus \dP} \bigoplus_{P\in E:\ \mathrm{closed\ points}} \kP^{\times} \onlong{\sum N_{\kP/k}} \kt.
\]
By the class field theory of curves over local field (\cite{Saito85b}, \cite{Yos03}), 
we have $V(E)/p^n \simeq \pi_1(E)^{\ab,\geo}_{\tor}/p^n$. 
Therefore, the above computations give the structure of $\pi_1(E)^{\ab}$.
\end{remark}


\def\cprime{$'$}
\providecommand{\bysame}{\leavevmode\hbox to3em{\hrulefill}\thinspace}
\providecommand{\href}[2]{#2}

\noindent
Toshiro Hiranouchi \\
Department of Mathematics, Graduate School of Science, Hiroshima University\\
1-3-1 Kagamiyama, Higashi-Hiroshima, 739-8526 Japan\\
Email address: {\tt hira@hiroshima-u.ac.jp}

\end{document}